\documentclass[french,11pt]{article}
\usepackage[utf8]{inputenc}
\usepackage[T1]{fontenc}
\usepackage{graphicx,amsmath,amsfonts,amssymb,amsthm,psfrag, color,framed}
\usepackage{comment}
\usepackage{calc}
\usepackage{wrapfig,float}
\usepackage[pagewise]{lineno}

\includecomment{comment} 

\newtheorem{thm}{Theorem}[section]

\newtheorem{lemma}[thm]{Lemma}
\newtheorem{proposition}[thm]{Proposition}

\newtheorem{rem}[thm]{Remark}
\numberwithin{equation}{section}

\title{Exit problem for Ornstein-Uhlenbeck processes:\\ a random walk approach.}

\begin{document}

\author{S. Herrmann and N. Massin\\[5pt]
\small{Institut de Math{\'e}matiques de Bourgogne (IMB) - UMR 5584, CNRS,}\\
\small{Universit{\'e} de Bourgogne Franche-Comt\'e, F-21000 Dijon, France} \\
\small{Samuel.Herrmann@u-bourgogne.fr}\\
\small{Nicolas.Massin@u-bourgogne.fr}
}
\maketitle

\begin{abstract}
In order to approximate the exit time of a one-dimensional diffusion process, we propose an algorithm based on a random walk. Such an algorithm so-called Walk on Moving Spheres was already introduced in the Brownian context. The aim is therefore to generalize this numerical approach to the Ornstein-Uhlenbeck process and to describe the efficiency of the method.
\end{abstract}
\section{Introduction}
Simulating the first exit time for a diffusion from a given domain is primordial since these  times appear in many domains. In mathematical finance, for instance, studying barrier options requires to estimate if the underlying stock price stays in a given interval. 
%
%
%
 In the simple Black-Scholes model, the distribution of the first exit time is therefore well-known. In more complex models corresponding to general diffusion processes, such an explicit expression is not available and requires the use of numerical approximations.

Several methods have been introduced in order to approximate first exit times. The classical and most common approximation method is the Euler–Maruyama scheme based on a time discretization procedure. The exit time of the diffusion process is in that case replaced by the exit time of the scheme. The approximation is quite precise but requires to restrict the study on a given fixed time interval on one hand and to describe precisely the probability for the diffusion to exit inbetween two consecutive nodes of the time grid on the other hand.
In this study, we aim to introduce a random walk in order to approximate the diffusion exit time from a given interval. Let us introduce $(X_t,\, t\ge 0)$ the unique solution of a stochastic differential equation:
\[
dX_t=b(t,X_t)\,dt+\sigma(t,X_t)\,dW_t,\quad t\ge 0,
\]
where $(W_t,\, t\ge 0)$ stands for a one-dimensional Brownian motion. Let us also fix some interval $I=[a,b]$ which strictly contains the starting position $X_0=x$. We denote by $\mathcal{T}$ the diffusion first exit time:
\[
\mathcal{T}=\inf\{t\ge 0:\ X_t\notin [a,b]\}.
\]
Our approach consists in constructing a random walk $(T_n,X_n)_{n\ge 0}$ on $\mathbb{R}_+\times \mathbb{R}$ which corresponds to a skeleton of the Brownian paths. In other words, the sequence $(T_n,X_n)$ belongs to the graph of the trajectory. Moreover we construct the walk in such a way that $(T_n,X_n)$ converges as time elapses towards the exit time and location $(\mathcal{T},X_{\mathcal{T}})$. It suffices therefore to introduce a stopping procedure in the algorithm to achieve the approximation scheme. Of course, such an approach is interesting provided that $(T_n,X_n)$ is easy to simulate numerically. For the particular Brownian case, the distribution of the exit time from an interval has a quite complicated expression which is difficult to use for simulation purposes (see, for instance \cite{sacerdote-telve-zucca}) whereas the exit distribution from particular time-dependent domains, for instance
the spheroids also called \emph{heat balls}, can be precisely determined. These time-dependent domains are characterized by their boundaries: 
 \begin{equation}
 \psi_\pm(t) = \pm\sqrt{t\log\left(\frac{d^2}{t}\right)}, \quad \text{for } t \in [0,d^2],
 \label{sphbm}
 \end{equation}
where the parameter $d>0$ corresponds to the size of the spheroid. The first time the Brownian motion path $(t, W_t)$ exits from the domain $\{(t,x):\ |x|\le \psi_+(t)\}$, denoted by $\tau$, is well-known. Its probability density function \cite{lerche} is given by
\begin{equation}
p(t)= \frac{1}{d \sqrt{2 \pi}} \sqrt{\frac{1}{t}\log\left(\frac{d^2}{t}\right)},\quad t\ge 0.
\end{equation}
It is therefore easy to generate such an exit time since $\tau$ and $d^2U^2e^{-N^2}$are identically distributed. Here $U$ and $N$ are independent random variables,  $U$ is uniformly distributed on $[0,1]$ and $N$ is a standard gaussian random variable.
Let us notice that the boundaries of the spheroids satisfy the following bound:
\begin{equation}
\vert \psi_\pm(t)\vert \leqslant \frac{d}{\sqrt{e}}, \quad \forall t\in [0,d^2].
\label{boundmb}
\end{equation}
This remark permits to explain the general idea of the algorithm. First we consider $(T_0,X_0)$ the starting time and position of the Brownian paths, that is $(0,x)$. Then we choose the largest parameter $d$ possible such that the spheroid starting in $(T_0,X_0)$ is included in the domain $\mathbb{R}_+\times [a,b]$. We observe the first exit time of this spheroid and its corresponding exit location, this couple is denoted by $(T_1,X_1)$. Due to the translation invariance of the Brownian motion, we can construct an iterative procedure, just considering $(T_1,X_1)$ like a starting time and position for the Brownian motion. So we consider a new spheroid included in the interval and $(T_2,X_2)$ shall correspond to the exit of this second spheroid and so on. Step by step we construct a random walk on spheroids also called WOMS algorithm (Walk On Moving Spheres) which converges towards the exit time and position $(\mathcal{T},W_{\mathcal{T}})$. This sequence is stopped as soon as the position $X_n$ is
close enough to the boundary of the considered interval. The idea of this algorithm lies in the definition of spherical processes and the walk on spheres introduced by Müller \cite{mul} and used in the sequel by Motoo \cite{mot} and Sabelfeld \cite{sab1} \cite{sab2}. It permits also in some more technical advanced way to simulate the first passage time for Bessel processes \cite{Deaconu-Herrmann}.

In this study, we focus our attention on a particular family of diffusions which is strongly related to the Brownian motion: the Ornstein-Uhlenbeck processes. The
idea is to use this link to adapt the Brownian algorithm in an appropriate way. This link implies changes on the time-dependent domains for which the exit problem can be expressed in a simpler way.
We present the random walk algorithm (WOMS) for the Ornstein-Uhlenbeck process, describe the approximation error depending on the stopping procedure and emphasize the efficiency of the method. We describe the mean number of generalized spheroids necessary to obtain the approximated exit time.

\section{The Ornstein-Ulhenbeck processes}

Let us first recall the definition of the Ornstein-Uhlenbeck process and present different essential properties which permit to link this diffusion to a standard Brownian motion.\\ 
Let $\theta \in \mathbb{R}^+$, $\sigma \in \mathbb{R}^+$, $\mu \in \mathbb{R}$. The Ornstein-Uhlenbeck process (O.U.) starting in $x_0$ with parameters $\theta$, $\mu$, and $\sigma$ is the unique solution of the following stochastic differential equation (SDE):
\begin{equation}
dX_t = - \theta(X_t - \mu)dt + \sigma dW_t,\quad t\ge 0,
\label{edsou}
\end{equation}
where $W$ stands for a standard one-dimensional Brownian motion.
Existence and uniqueness for equation \eqref{edsou} can be easily deduced from a general statement concerning SDE, see for instance Revuz, Yor, Chap. IX \cite{rev}. Let us just recall this result.
\begin{proposition}
\label{pathunique}
Consider the following stochastic differential equation 
\begin{equation}
dX_t = b(t,X_t)dt + \sigma(t,X_t) dW_t,\quad t\ge 0.
\label{existence}
\end{equation}
If there exists a Borel function $\rho : ]0, + \infty[ \rightarrow ]0, + \infty[$ satisfying $\int_0^{+\infty} \frac{dx}{\rho(x)} = +\infty$ and such that
\begin{equation*}
\vert \sigma(s,x) - \sigma(s,y)\vert^2 \leqslant \rho(\vert x-y  \vert),\quad  \forall x, y \in ]0,+\infty[,\quad  \forall t \in \mathbb{R}^+.
\end{equation*}
and if, for each compact set $H$ and each $t \geqslant 0$, there exists a constant $K_t>0$ such that 
\begin{equation*}
\vert b(s,x) - b(s,y)\vert \leqslant K_t \vert x-y  \vert,\quad  \forall x, y \in H, \quad s \leqslant t
\end{equation*}
then pathwise uniqueness holds for equation \eqref{existence}.
\end{proposition}
Since obviously the drift and diffusion coefficients of the O.U. process satisfy the hypotheses of Proposition \ref{pathunique}, pathwise uniqueness holds for \eqref{edsou}. Let us now present an explicit expression of this solution. The Ornstein-Uhlenbeck process can be written as a stochastic integral with respect to the Brownian motion:
\begin{equation}
X_t = X_0e^{-\theta t} + \mu(1 - e^{-\theta t}) + \sigma e^{-\theta t} \displaystyle{\int_0^t e^{\theta s} dW_s}, t \geqslant 0.
\end{equation}
Levy's theorem permits to replace the stochastic integral by a time-changed Brownian motion. We obtain therefore another expression for the process which is more handy to manipulate.\\
Since $\theta>0$, there exists a standard Brownian motion $(V_t)_{t\ge 0}$ such that 
\begin{equation}
X_t = X_0 e^{-\theta t} + \mu(1 - e^{-\theta t}) +\frac{\sigma e^{-\theta t}}{\sqrt{2\theta}} V_{\displaystyle{e^{2\theta t} - 1}}.
\label{eqou} 
\end{equation}
This simplified expression is a crucial tool for the construction of the algorithm in the exit problem framework as it clearly appears in the forthcoming statements.

\begin{rem}
In following computations, we put  $\mu = 0$. This restriction is only motivated by notational simplification and the study can easily be extended to the general case.
\end{rem}

Let us now describe how such a strong relation between the Brownian motion and the Ornstein-Uhlenbeck process permits to emphasize a time-dependent domain of $\mathbb{R}$ whose exit time can be easily and exactly simulated.

\section{Exit time of generalized spheroids}
Let us consider the spheroids defined by the boundaries $\psi^{\pm}(t)$ in \eqref{sphbm}. We recall that the Brownian exit problem of a such a spheroid is completely explicit, so that the simulation of the exit time $\tau$ is rather simple. Due to the symmetry property of the spheroid, the conditional probability distribution of the exit location $W_{\tau}$ given $\tau$ is equal to $\frac{1}{2}\,\delta_{\psi^+(\tau)}+\frac{1}{2}\,\delta_{\psi^-(\tau)}$.
For the Ornstein-Uhlenbeck process, we can obtain some similar information due to the strong relation with the Brownian motion.\\ 
Let us introduce two new boundaries defined by:
\[
\psi_{\rm OU}^\pm(t,x)=e^{-\theta t}\left(\frac{\sigma}{\sqrt{2 \theta}} \psi_{\pm}(e^{2 \theta t} -1) + x\right),
\]
where $\theta$ and $\sigma$ correspond to the parameters of the O.U-process $(X_t)_{t\ge 0}$ in \eqref{edsou}. We call \emph{generalized spheroid} the domain defined by these boundaries.\\
We introduce the exit time $\tau_{\rm OU} =\inf\{t >0:\  X_t \notin [\psi_{\rm OU}^-(t,x),\psi_{\rm OU}^+(t,x)]\}$. 

\begin{proposition}
\label{exitimou}
Let $\tau=\inf\{t >0:\ V_t \notin [\psi_-(t),\psi_+(t)]\}$ the first time the Brownian motion $(V_t)_{t\ge 0}$ defined in \eqref{eqou} exits from the spheroid. Then the exit time $\tau_{OU}$ satisfies:

\begin{equation}
\tau_{\rm OU}=\frac{\log(\tau + 1)}{2\theta} \text{ } a.s.
\label{relou}
\end{equation}
\end{proposition}
 
\begin{proof}
Using both the definition of $\tau_{\rm OU}$ and the expression of $X_t$ with respect to the Brownian motion $V_t$, we obtain 
\begin{align*}
\tau_{\rm OU} &= \inf\left\lbrace t >0:\ X_t \notin [\psi_{\rm OU}^-(t,x),\psi_{\rm OU}^+(t,x)]\right\rbrace\\
&=\inf\left\lbrace t >0:\ xe^{-\theta t} + \frac{\sigma}{\sqrt{2\theta}}e^{-\theta t} V_{e^{2\theta t} -1} \notin [\psi_{\rm OU}^-(t,x),\psi_{\rm OU}^+(t,x)]\right\rbrace \\
&=\inf\left\lbrace t >0:\frac{\sigma}{\sqrt{2\theta}}e^{-\theta t} V_{e^{2\theta t} -1} \notin [\psi_{\rm OU}^-(t,x) - xe^{-\theta t},\psi_{\rm OU}^+(t,x) - xe^{-\theta t}]\right\rbrace\\
&=\inf\left\lbrace \frac{\log(u+1)}{2\theta} >0:\  V_u \notin [\psi_{-}(u),\psi_{+}(u)] \right\rbrace =\frac{\log(\tau + 1)}{2\theta}.
\end{align*}
\end{proof}

This statement is a crucial tool for simulation purposes. It permits first to simulate a  Brownian exit time from a spheroid, then to use Proposition \ref{exitimou} to obtain the O.U. exit time from the generalized spheroid. Let us notice that the shape of the generalized spheroid depends on the O.U. starting position. Therefore, if we define a WOMS, the shape of the spheroids will change at each step of the algorithm. In the Brownian motion context, the spheroids are symmetric and their extremas can be computed easily. This important advantage  permits to compute easily the maximal size of the spheroids included in the interval $[a,b]$ and is not fulfilled in the O.U. case. It is therefore an harder work to determine the optimal size of the generalized spheroid. This can be achieved by finding an upper-bound for the upper boundary and a lower-bound  for the lower boundary. As a consequence,  we determine a parameter characterizing the generalized spheroid which guaranties that it remains fully contained in the interval $[a,b]$. Since the bounds are quite rough, the boundaries of the generalized spheroid are unfortunately not tangent to the interval bounds. 
The algorithm shall be therefore a little slowed down.

\begin{proposition}
\label{dou}
Let $\gamma > 0$, and $x \in [a,b]$ the starting point of the spheroid, that is $\psi^{\pm}_{OU}(0,x) = x$. Let us set $a_{\gamma,x} = a+ \gamma(x-a)$ and $b_{\gamma,x} = b - \gamma (b-x)$.  We define
\begin{equation}
d=\left \{
\begin{array}{c @{\text{ if }} l}
    \sqrt{2 \theta e} \min\left(\frac{(b_{\gamma,x}-x)}{\sigma},\frac{2(x-a_{\gamma,x})}{\sqrt{\sigma^2 + 4\theta e\,x(x-a_{\gamma,x})} + \sigma}\right) &  x\geqslant 0\\
    \sqrt{2 \theta e} \min\left( \frac{(x-a_{\gamma,x})}{\sigma},\frac{2(b_{\gamma,x}-x)}{\sqrt{\sigma^2- 4\theta e\,x(b_{\gamma,x}-x)} + \sigma}\right) & x\leqslant 0
\end{array}
\right.
\label{paramou}
\end{equation}
For such a choice of parameter, the generalized spheroid is fully contained in the interval $[a_{\gamma,x},b_{\gamma,x}]$.
\end{proposition}

In the following statements we denote by $d_x$ the parameter associated to the spheroid with initial point $x$.

\begin{proof}
Let us first consider the case: $x>0$. Combining the upper bound of the function $\psi^+$ presented in \eqref{boundmb} and the definition of $\psi_{\rm OU}$, we obtain

\begin{equation}
-\frac{\sigma d}{\sqrt{2 \theta e}} + \frac{x}{\sqrt{ 1+ d^2}} \leqslant\psi_{\rm OU}^-(t,x) \leqslant \psi_{\rm OU}^+(t,x) \leqslant \frac{\sigma d}{\sqrt{2 \theta e}} + x.
\label{boundou}
\end{equation}
We keep the upper bound found previously and focus on the lower bound:
\begin{equation}
\psi_{\rm OU}^-(t,x) \geqslant -\frac{\sigma d}{\sqrt{2 \theta e}} + \frac{x}{\sqrt{ 1+ d^2}} \geqslant -\frac{\sigma d}{\sqrt{2 \theta e}} + x(1 - \frac{d^2}{2}).
\label{dvlptinfOU}
\end{equation}
The determination of a convenient choice for the parameter $d>0$ requires to find the positive solution of the equation $P(d)=0$ where 
\begin{equation*}
P(d) = x \frac{d^2}{2} + \frac{\sigma}{\sqrt{2 \theta e}}d +(a_{\gamma,x}-x).
\end{equation*}
Consequently we obtain
\begin{equation*}
d_l = \frac{1}{x}\sqrt{\frac{\sigma^2}{2 \theta e}+ 2x(x-a_{\gamma,x})} - \frac{\sigma}{x\sqrt{2 \theta e}}.
\end{equation*}
The identification with the upper bound gives us 
\begin{equation}
d_u = (b_{\gamma,x}-x)\frac{\sqrt{2 \theta e}}{\sigma}.
\end{equation}
Hence setting $d = \min(d_u, d_l)$ permits the generalized spheroid to belong to the interval $[a_{\gamma,x},b_{\gamma,x}]$.

The case $x<0$  uses similar arguments since we observe a symmetry with respect to the origin between the generalized spheroid starting in $x$ and the one starting in $-x$. We use the results previously computed for $\vert x\vert$ and $[-b_{\gamma,x},-a_{\gamma,x}]$ which leads to the statement. The case $x=0$ is simple to handle with, since the previous boundaries \eqref{boundou} become
\begin{equation*}
-\frac{\sigma d}{\sqrt{2 \theta e}} \leqslant\psi_{\rm OU}^-(t,0) \leqslant \psi_{\rm OU}^+(t,0) \leqslant \frac{\sigma d}{\sqrt{2 \theta e}}.
\end{equation*}
It suffices to set $d = \displaystyle{\frac{\sqrt{2 \theta e}}{\sigma}}\min(|a_{\gamma,0}|,b_{\gamma,0})$, which corresponds to the limit case as $x$ tends to $0$ in both results previously established.

\end{proof}

\section{WOMS for the Ornstein-Uhlenbeck processes}
Let us now present the approximation procedure of the Ornstein-Uhlenbeck exit time from a given interval $[a,b]$. This algorithm is based on a walk on generalized spheroids (WOMS) described in the previous section.

\begin{center}
ALGORITHM  (O.U. WOMS)
\end{center}
\fbox{\begin{minipage}{\textwidth}

\textit{Initialization:} Let: $X_0 = x_0$, $\mathcal{T}_{\epsilon} = 0$

\textit{From step $n$ to step $n+1$:}

While $X_n \leqslant b - \epsilon$ and $X_n \geqslant a + \epsilon$ do

\begin{itemize}
\item[$\bullet$] Generate the Brownian exit time from the spheroid with parameter $d_{X_n}$ defined in \eqref{paramou}. We denote this stopping time by $\tau_{n+1}$.

\item[$\bullet$] We set $\tau^{OU}_{n+1} =\frac{\log(\tau_{n+1} + 1)}{2\theta}$.

\item[$\bullet$] Generate a Bernoulli distributed r.v.  $\mathcal{B}\sim\mathcal{B}(\frac{1}{2})$, if $\mathcal{B}=1$ then set $X_{n+1} = \psi_{OU}^-(\tau^{OU}_{n+1},X_n)$ otherwise set $X_{n+1} = \psi_{OU}^+(\tau^{OU}_{n+1},X_n)$.

\item[$\bullet$] $\mathcal{T}_{\epsilon} \leftarrow \mathcal{T}_{\epsilon} + \tau_{n+1}^{OU}$.
\end{itemize}

\textit{Outcome:} $\mathcal{T}_{\epsilon}$ the approximated O.U.-exit time from the interval $[a,b]$.
\end{minipage}}

\begin{figure}[H]
   \centerline{\includegraphics[scale=0.3]{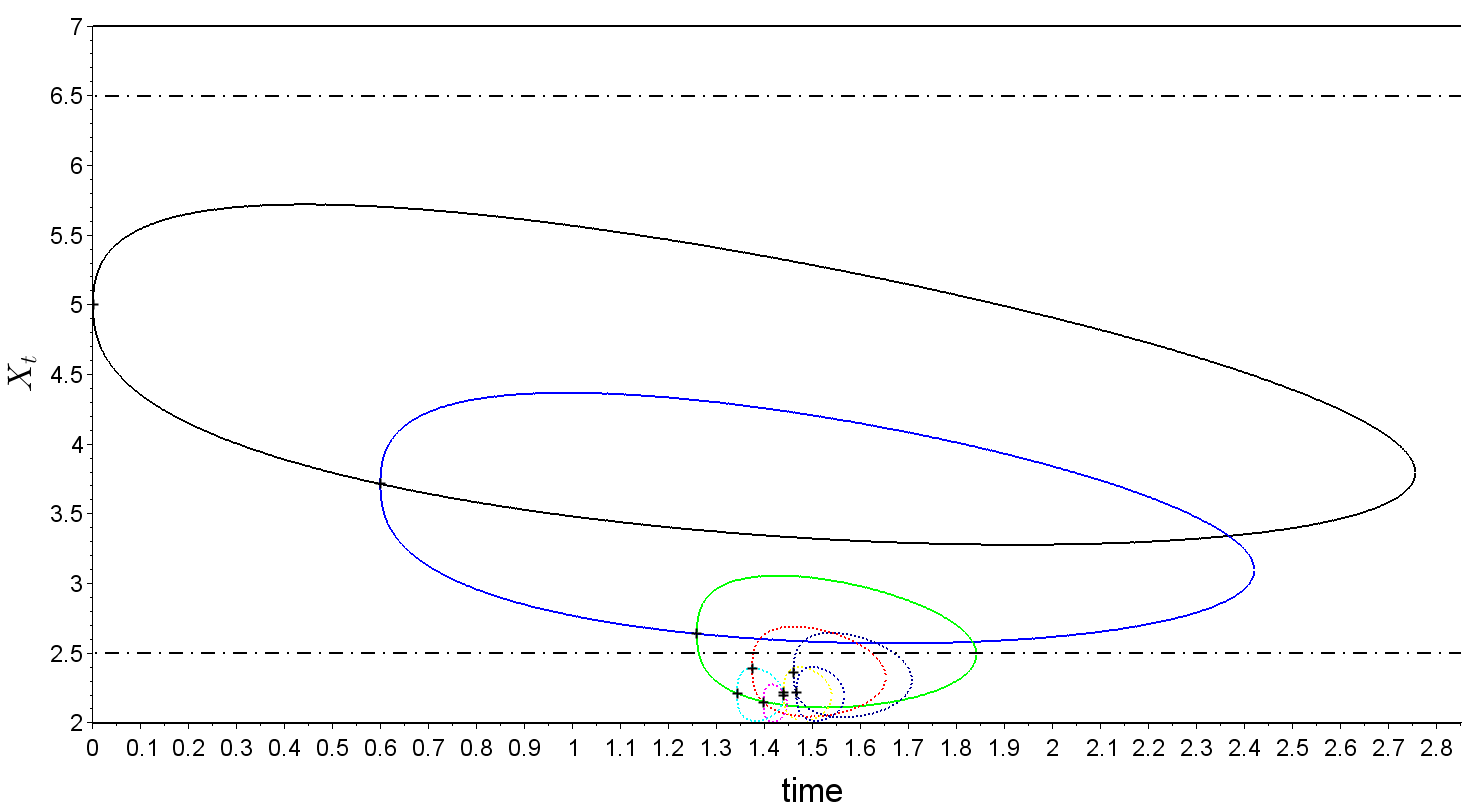}}
   \caption{\small A sample of  the algorithm for the O.U. exit time with parameters $\theta =0.1$ and $\sigma = 1$. We observe the walk on spheres associated with the diffusion process starting at $x=5$ and moving in the interval $[2,7]$. The algorithm corresponding to $\epsilon = 0,5$ is represented by the plain style spheroids whereas the case $\epsilon = 10^{-3}$ corresponds to the whole sequence of spheroids. In both cases we set $\gamma = 10^{-6}$ .}
\end{figure}

\begin{figure}[H]
   \includegraphics[width=\textwidth]{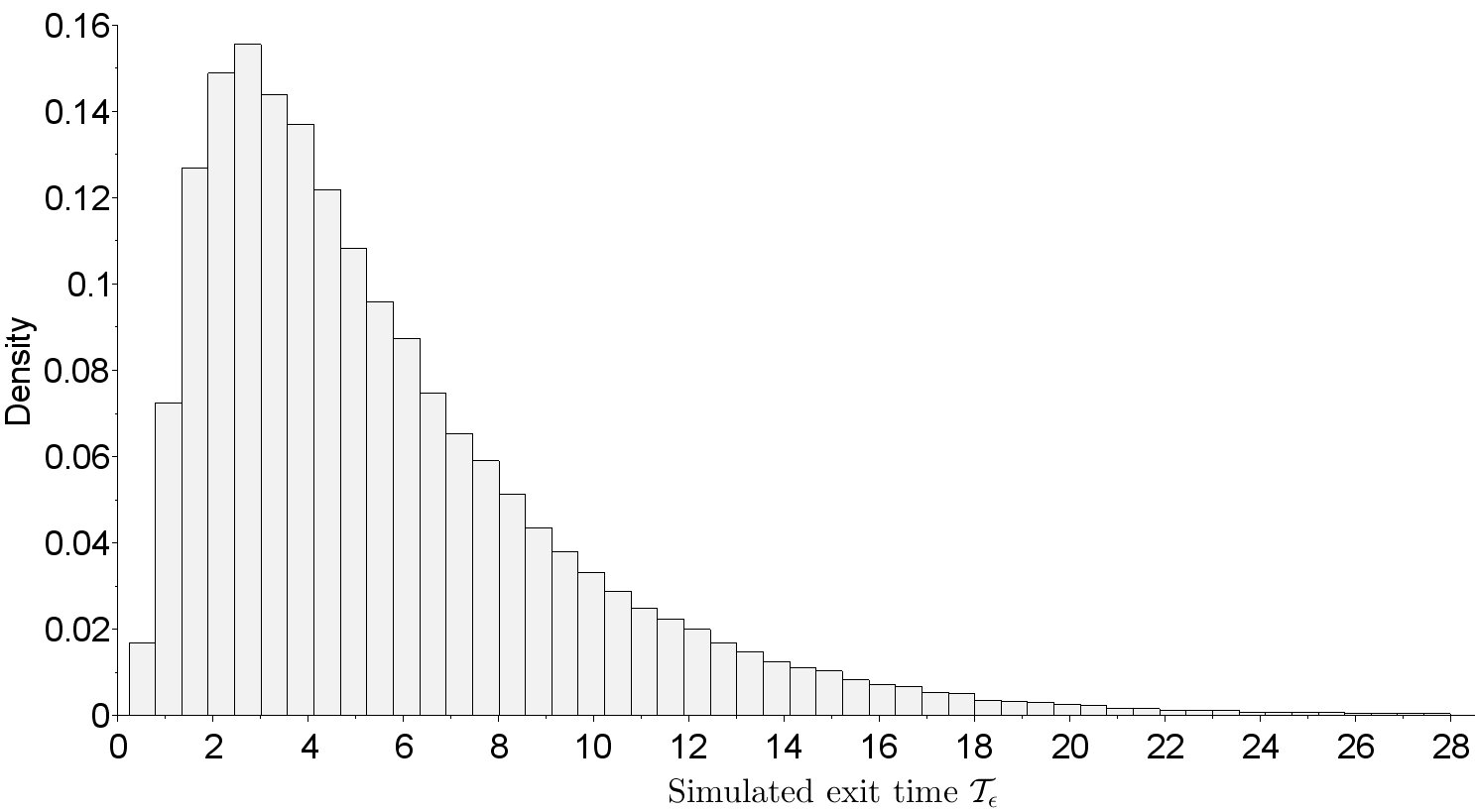}
   \caption{Histogram of the outcome variable for the O.U. with parameters $\theta =0.1$ and $\sigma = 1$ when the stopped diffusion process starts at 5 and involves in the interval [2,7] with $\epsilon = 10^{-3}$ and $\gamma = 10^{-6}$.}
\end{figure}

The CPU efficiency of such an algorithm shall be compared to the efficiency of classical approaches in the exit time approximation framework. Let us consider a particular situation:  the exit time from the interval  $[3,5]$ for the Ornstein-Uhlenbeck process  starting in $4$ with $\theta = 5$ and $\sigma =7$.  We use an improved Euler method based on the correction by means of the sharp large deviations estimate of the exit probability. Such a method takes into account the probability for the diffusion path to exit inbetween two neighboring gridpoints (see the procedure described in \cite{baldi-caramellino-iovino}). The simulation of $100\,000$ samples with the step size $10^{-4}$ requires 64,7 seconds for this improved Euler method whereas the WOMS algorithm presented in this paper requires about 2,19 seconds for the corresponding choice $\epsilon = 10^{-2}$ (here $\gamma = 10^{-6}$).\\

Even if the study presented here concerns the exit time of some given interval $[a,b]$ denoted by $\tau_{[a,b]}$, let us just mention the possible link with first passage times (FPT). Intuitively for negative $a$ with large value $|a|$, the exit time of the interval can be approximated by the first passage time of the level $b$ denoted by $\tau_b$ i.e. $\lim_{|a|\to\infty} \mathbb{P}(\tau_{[a,b]}=\tau_b)=1$. 
Several approaches permit to describe quite precisely the probability distribution of the Ornstein-Uhlenbeck FPT. In Figure \ref{fig:compar}, we illustrate that the distributions of both the exit time (histogram) and the first passage time (p.d.f.) present a thight fit. The histogram corresponds to the exit time obtained for an OU process starting in $-3$ with coefficients $\theta =1$ and $\sigma = 1$ and observed on the interval with bounds $a=-10$ and $b=-1$. The curve corresponds to a numerical approximation of the first passage time density presented by Buonocore, Nobile and Ricciardi in \cite{buonocore-al}. An other approximation procedure for the FPT simulation is proposed by Herrmann and Zucca in \cite{herrmann-zucca-exact}: it consists in simulating exactly the PFT of a slightly modified diffusion process. This modified diffusion has the following property: its drift term is bounded and coincides with the Ornstein-Uhlenbeck drift on the interval $[a,b]$ with $|a|$ large. Numerical comparisons permit to observe that the simulation of the exit time with the WOMS algorithm is highly more efficent than the method proposed in \cite{herrmann-zucca-exact}: the simulation of a sample of size $100\,000$ takes a total of 3,7 seconds of CPU time with the first method and 197,2 seconds with the second one. Here the OU-process starts in $-3$ with coefficients $\theta =1$ and $\sigma = 1$ and is observed on the interval $[-10,-1]$.
\begin{figure}[H]
   \centerline{\includegraphics[width=\textwidth]{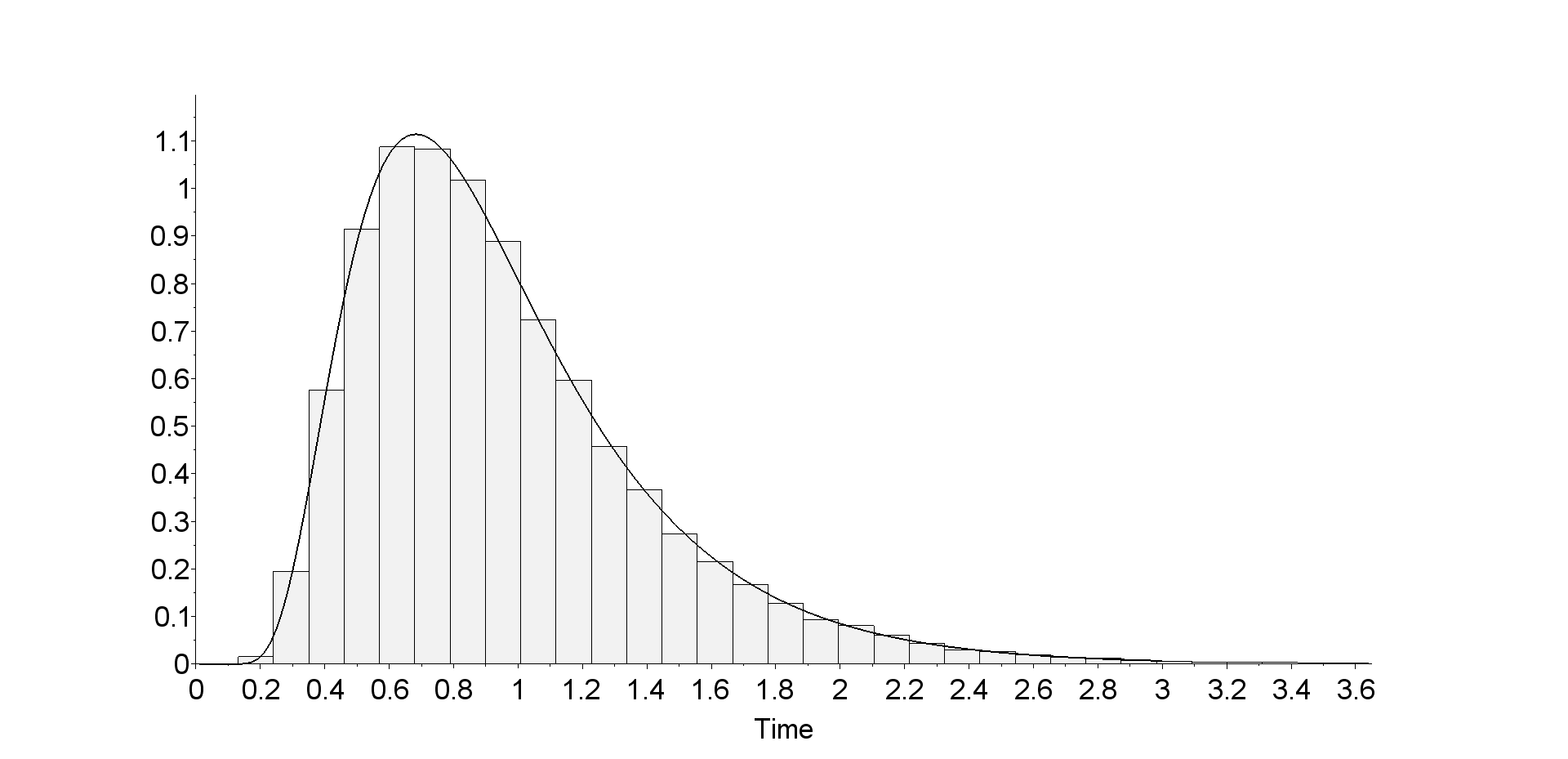}}
   \caption{\small Histogram of the approximated first exit time of the interval $[a,b]$ using the WOMS algorithm and approximated p.d.f. of the first passage time through the level $b$ (curve). Here $X_0=-3$, $\theta =1$, $\sigma = 1$ and $[a,b]=[-10,-1]$.}
   \label{fig:compar}
\end{figure}
Let us now describe the WOMS algorithm for the Ornstein-Uhlenbeck process and especially emphasize its efficiency through theoretical results.  We study how the strong relation between our process and the Brownian motion affects the statements obtained in the Brownian motion case. Let us just recall that the efficiency of the walk on spheres in the particular Brownian case is quite strong: the averaged number of steps is of the order $|\log(\epsilon)|$  (see for instance \cite{binder-braverman}, for an overview of the convergence rate). In the Ornstein-Uhlenbeck case, we reach a similar efficiency result. 

\subsubsection*{Average number of steps}

\begin{thm}
\label{meanou}
Let $N_{\epsilon}$ be the random number of steps observed in the algorithm. Then there exist a constant $\delta>0$ and $\epsilon_0 > 0$ such that 

\begin{equation}
\mathbb{E}[N_{\epsilon}] \leqslant \delta|\log(\epsilon)|, \quad \forall \epsilon \leqslant \epsilon_0.
\label{avou}
\end{equation}
\end{thm}

\begin{figure}[H]
  \centerline{\includegraphics[width=7.5cm]{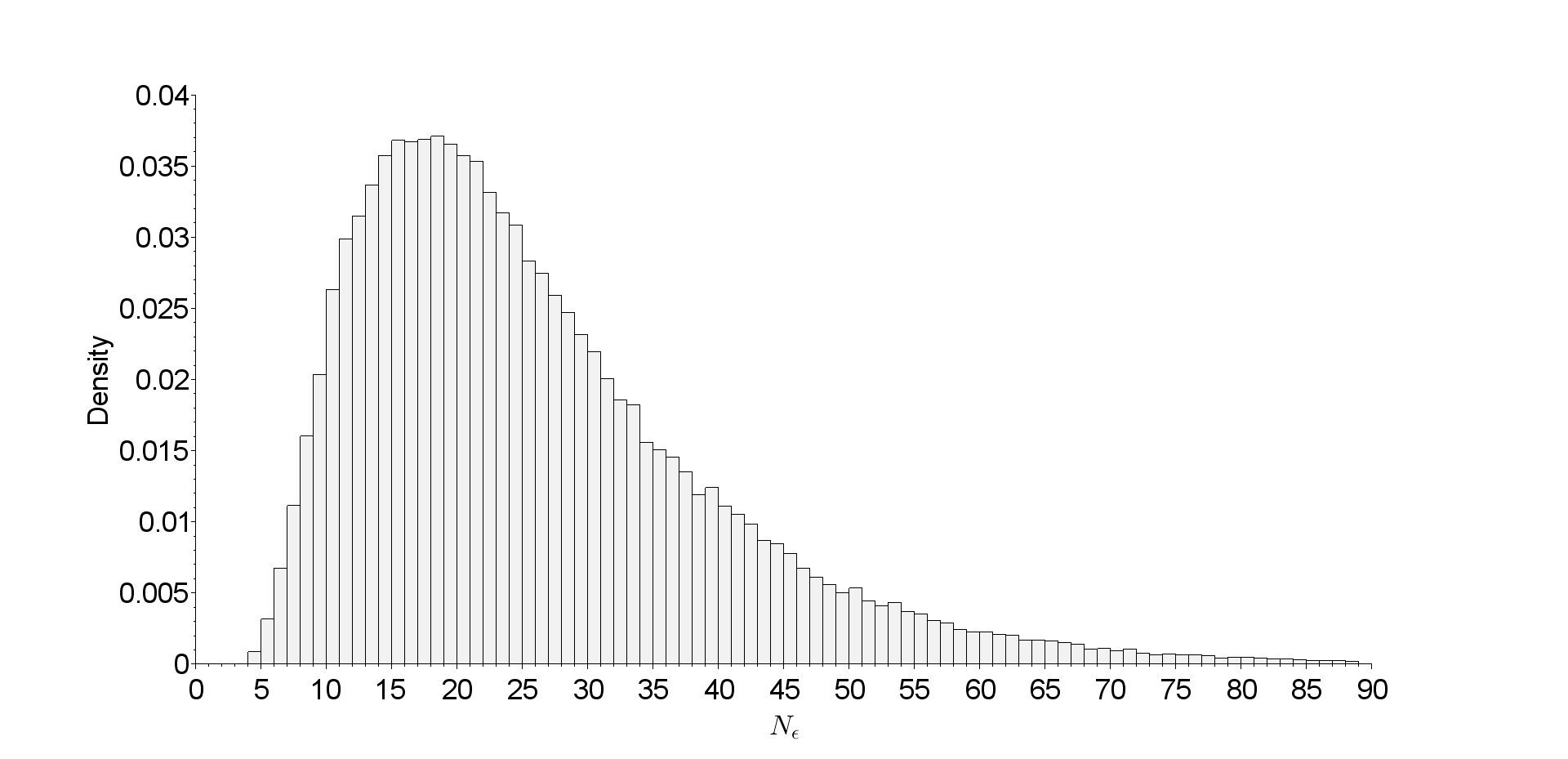}  \includegraphics[width=7.5cm]{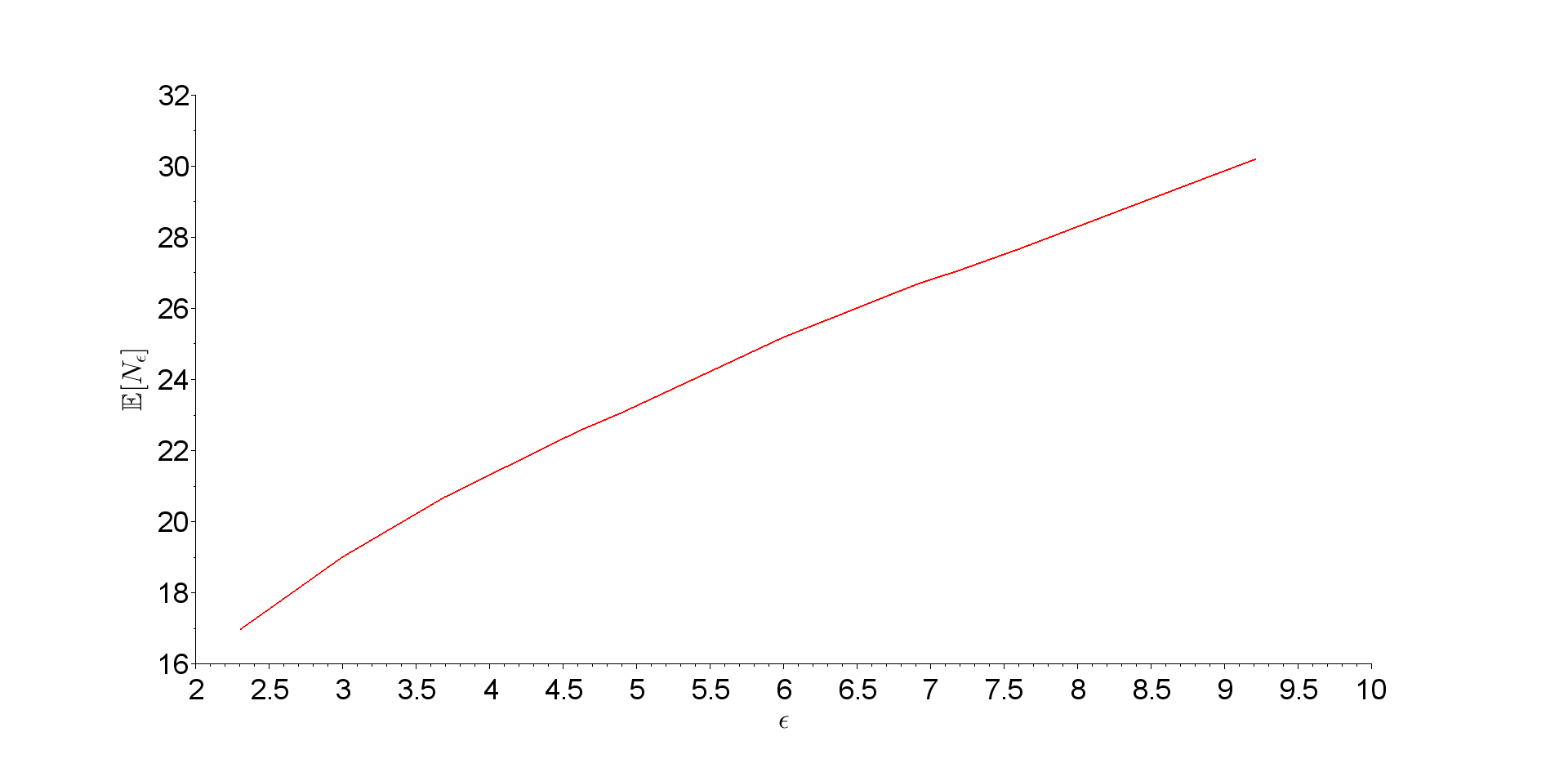}}
   \caption{\small Simulation of the O.U. exit time from the interval $[2,7]$. The starting position is $X_0=5$ and the parameters are given by $\theta =0.1$, $\sigma = 1$ and $\gamma = 10^{-6}$. Histogram of the number of steps observed for $\epsilon = 10^{-3}$ (left) and average number of steps versus $\epsilon$  (right, in logarithmic scale). }
   \label{sinresult}
\end{figure}

The statement is similar to the Brownian motion case, and the proofs are based on similar arguments.
To prove this statement, we introduce a result coming from the potential theory and using Markov chains.\\
Let us consider a Markov chain $(X_n)_{n \in \mathbb{N}}$ defined on a state space $I$  decomposed into two distinct subsets $K$ and $\partial K$, $\partial K$ being  the  so-called frontier. Let us define $N= \inf\{n \in \mathbb{N}, X_n \in \partial K\}$ the hitting time of $\partial K$.
We assume that $N$ is a.s. finite, then the following statement holds:
\begin{proposition} 
\label{potential}
If there exists a function $U$ s.t. the sequence $(U(X_{n\wedge N}))_{n \in \mathbb{N}}$ is non negative and if the sequence $(U(X_{n\wedge N}) + {n\wedge N})_{n \in \mathbb{N}}$ represents a super-martingale adapted to the natural filtration of the considered Markov chain $(X_n)$, then  

\begin{equation*}
\mathbb{E}_x[N] \leqslant U(x), \quad \forall x \in K.
\end{equation*}

\end{proposition}

The proof of this classical upper-bound is left to the reader, it is essentially based on the optimal stopping theorem and on the monotone convergence theorem (see, for instance,\cite{norris}, p139). 
\begin{proof}[Proof of Theorem \ref{meanou}]~\\
\textbf{Step 1.} Let us first introduce a function $u$ which plays an important role in the construction of a super-martingale linked to the random walk.\\
We consider the following differential equation:\\
\begin{equation}
\frac{\sigma^2}{2}u'' - \theta x u' = \frac{-1}{(x-a)^2(x-b)^2}, \quad \text{for } x\in ]a,b[.
\label{edou}
\end{equation}
This second order differential equation can be solved in a classical way. Let us first solve the related homogeneous equation: we obtain
\begin{equation*}
u'(x) = C(x)e^{\frac{2\theta}{\sigma^2}x}.
\end{equation*} 
The method of variation of parameters leads to
\begin{equation*}
C(x) = -\frac{2}{\sigma^2}\int^x_0 \frac{e^{-\frac{2\theta}{\sigma^2}s}}{(s-a)^2(s-b)^2}ds.
\end{equation*}
Integrating $u'$ one more time implies an explicit expression of one particular solution \eqref{edou}.
\begin{equation}
u(x) = -\frac{2}{\sigma^2} \int^x_0 e^{\frac{2\theta}{\sigma^2}u}\int^u_0 \frac{e^{-\frac{2\theta}{\sigma^2}s}}{(s-a)^2(s-b)^2}ds du, \quad \text{for } x\in]a,b[.
\label{solou}
\end{equation}
\textbf{Step 2.} We consider now the sequence $(T_n, X_n)_{n \in \mathbb{N}}$ of cumulative exit times, i.e. 
\begin{equation}
T_n = \displaystyle{\sum_{k=1}^n \tau_k^{\rm OU}}
\label{defTnOU}
\end{equation}
 and exit location given by the WOMS algorithm for the Ornstein-Uhlenbeck process.

Let us introduce $Z_n = u( X_n) + cn$ where $c$ is a positive constant (which shall be determined in the following calculus) and $u$ is the function detailed in Step 1 of the proof. We shall prove that this process is a super-martingale with respect to the filtration $(\mathcal{F}_{T_n})_{n \in \mathbb{N}}$ induced by $(\mathcal{F}_t)$, the natural filtration of the Brownian motion $(V_t)_{t\geqslant 0}$ enlightened in \eqref{eqou}.\\
By Itô's formula we obtain

\begin{align}
\mathbb{E}[Z_{n+1} - Z_n \vert \mathcal{F}_{T_n}] &= \mathbb{E}[M_{n+1} - M_n \vert \mathcal{F}_{T_n}]\nonumber \\ & + \mathbb{E}\left.\left[\int^{T_{n+1}}_{T_n} \frac{\sigma^2}{2}u"(X_s) - \theta X_s u'(X_s)ds \right\vert \mathcal{F}_{T_n}\right] +c \nonumber\\
&= \mathbb{E}\left.\left[\int^{T_{n+1}}_{T_n} \frac{-1}{D_{[a,b]}(X_s)^2}ds \right\vert \mathcal{F}_{T_n}\right] + c.
\label{diffz}
\end{align}
where $(M_n)_{n \in \mathbb{N}}= \left( \int^{T_n}_0 \sigma u'(X_s)dW_s\right)_{n \in \mathbb{N}}$ is a martingale and $D_{[a,b]}(x) = (x-a)(b-x)$ for $x\in [a,b]$.
Remark now that

\begin{equation}
\Xi(X_n):=\mathbb{E}\left.\left[\int^{T_{n+1}}_{T_n} \frac{-1}{D_{[a,b]}(X_s)^2}ds \right\vert \mathcal{F}_{T_n}\right] = \mathbb{E}\left.\left[\int^{\tau^{\rm OU}_{n+1}}_0 \frac{-1}{D_{[a,b]}(\tilde{X}_s)^2}ds \right\vert \mathcal{F}_{T_n}\right]
\end{equation}
where $ \tilde{X}_s:=X_{T_n+s}$ has the same distribution as the Ornstein-Uhlenbeck starting in $X_n$. We now upper bound this term: we consider in a first time that $X_{n}$ is positive.  By Proposition \ref{dou} we are then allowed to compute the corresponding coefficient $d_{X_n}$ which we denote by $d_n>0$ for notation simplicity. Let us fix some parameter $\Delta\in]0,1[$.\\
\textbf{\underline{First case:}}  $d_n\leqslant \Delta$, that is satisfied either if
 
\begin{equation}
0<(b_{\gamma} -X_{n})\frac{\sqrt{2 \theta e}}{\sigma} \leqslant \Delta 
\end{equation}
or
\begin{equation}
0<\frac{2(X_{n}-a_{\gamma})}{\sqrt{\frac{\sigma^2}{2 \theta e}+ 2X_{n}(X_{n}-a_{\gamma})} + \frac{\sigma}{\sqrt{2 \theta e}}} \leqslant \Delta
\end{equation}
with $b_{\gamma} = b_{\gamma,X_{n}}$ and $a_{\gamma} = a_{\gamma,X_{n}}$.\\
We first consider that $X_{n}$ is close enough to $b_{\gamma}$. Using \eqref{boundou}, we have for any $t \in \text{Supp}\,(\psi_{OU}^{\pm}) = \left[0, \frac{\log(1+ d_n^2)}{2 \theta}\right]$:
\begin{equation*}
b - \psi_-^{OU}(t,X_n) \leqslant  b - \frac{X_{n}}{\sqrt{1+ d_n^2}} + \frac{\sigma d_n}{\sqrt{2 \theta e}} \leqslant b - X_{n}\left(1-\frac{d_n^2}{2}\right)+ \frac{\sigma d_n}{\sqrt{2 \theta e}}.\\
\end{equation*}
Since  $d_n\le \Delta<1$, we have $d_n^2\le d_n$. Moreover  $X_n\le b_\gamma$ so that
\begin{align*}
b - \psi_-^{OU}(t,X_n) &\leqslant b - X_{n}\left(1-\frac{d_n}{2}\right)+ \frac{\sigma d_n}{\sqrt{2 \theta e}}\\
 &= b - X_{n} + d_n\left(\frac{X_{n}}{2} + \frac{\sigma}{\sqrt{2 \theta e}}\right)\\
 &\leqslant (b - X_{n})\left(\frac{b\sqrt{2\theta e}}{2\sigma} +2\right)=: (b-X_{n})\beta.
\end{align*}
The last upper-bound uses the definition of $d_n$ presented in Proposition \ref{dou} Hence we have 
\begin{equation*}
D_{[a,b]}(\tilde{X}_s)\leqslant \beta(b-a)(b-X_{n}).
\end{equation*}
We then write, using the fact that $\tau^{\rm OU}_{n+1}$ is independent of $\mathcal{F}_{T_n}$,

\begin{align*}
\Xi(X_n) &\leqslant \mathbb{E}\left.\left[\int^{\tau^{\rm OU}_{n+1}}_0 \frac{-1}{\beta^2(b-a)^2(b-X_{n})^2}ds \right\vert \mathcal{F}_{T_n}\right]\\
&= \frac{-1}{\beta^2(b-a)^2(b-X_{n})^2} \, \mathbb{E}[\tau^{\rm OU}_{n+1}] \\
&=\frac{-1}{2\theta\beta^2(b-a)^2(b-X_{n})^2} \, \mathbb{E}[\log(1 + \tau_n)],
\end{align*}
where $\tau_n$ denotes the exit time for Brownian motion from the spheroid of parameter $d_n$. If $\tau$ denotes the Brownian exit time of the generalized spheroid of normalized size ($d=1$), then the scaling property of Brownian motion implies that $\tau_n$ and $d^2_n \tau$ are identically distributed. Hence, noticing that $\tau \leqslant 1$ and recalling that $d_n^2 \leqslant 1$, we obtain

\begin{align*}
\Xi(X_n) &\leqslant \frac{-1}{2\theta\beta^2(b-a)^2(b-X_{n})^2} \, \mathbb{E}[\log(1 + d_n^2 \tau)]\\
&\leqslant\frac{-d_n^2}{4\theta\beta^2(b-a)^2(b-X_{n})^2} \, \mathbb{E}[\tau].
\end{align*} 
In the considered case, we know that  
\begin{equation}
d_n  = (b-X_{n})\frac{\sqrt{2 \theta e}}{\sigma}
\end{equation}
which implies
\begin{equation}
\Xi(X_n) \leqslant\frac{-e}{2\sigma^2\beta^2(b-a)^2}\mathbb{E}[\tilde{\tau}_1].
\label{petitdn}
\end{equation} 
In the other case ($X_{n}$ close to $a$) the arguments already used just above lead to a similar upper-bound. We observe for any  $t \in \left[0, \frac{\log(1+ d_n^2)}{2 \theta}\right]$:
\begin{align*}
\psi_+^{OU}(t,X_n) - a &\leqslant X_{n} + \frac{\sigma d_n}{\sqrt{2 \theta e}} - a\\
&\leqslant (X_{n}-a)\left(1 + \frac{2\sigma}{\sqrt{\sigma^2+4 \theta eX_{n}(X_{n}-a)} + \sigma}\right) \leqslant 2(X_{n}-a).
\end{align*}
This upper bound leads to the same result as \eqref{petitdn} just replacing $\beta$ by another positive constant $\tilde{\beta}$. Combining both inequalities, for $d_n$ smaller than $\Delta$, we get 
\begin{equation}
 \Xi(X_n) \leqslant \frac{-1}{\sigma^2\max(\tilde{\beta},\beta)^2} \, \mathbb{E}[\tau].
 \label{boundbeta}
\end{equation}
\textbf{\underline{Second case:}} $d_n > \Delta$\\
In this case, we use the upper-bound:
\begin{equation}
D_{[a,b]}(\tilde{X}_s)\leqslant (b - a)^2.
\end{equation}
We deduce
\begin{align}
\Xi(X_n) &\leqslant \mathbb{E}\left.\left[\int^{\tau^{\rm OU}_{n+1}}_0 \frac{-1}{(b-a)^4}ds \right\vert \mathcal{F}_{T_n}\right]\nonumber\\
&\leqslant \frac{-1}{2\theta(b-a)^4} \, \mathbb{E}[\log(1 + \Delta^2 \tilde{\tau}_1)]\leqslant\frac{-\Delta^2}{4\theta(b-a)^4} \, \mathbb{E}[\tau].
\label{boundelta}
\end{align}
Both inequalities \eqref{boundbeta} and \eqref{boundelta} suggest the existence of  a constant $\tilde{c}>0$ such that $\Xi(X_n) \leqslant -\tilde{c}$.\\
Finally, using the symmetry property of the considered spheroid, the case $x$ negative is treated similarly, leading to a positive constant $c$ such that
\begin{equation}
\Xi(X_n) \le \mathbb{E}\left.\left[\int^{T_{n+1}}_{T_n} \frac{-1}{D_{[a,b]}(X_s)^2}ds \right\vert \mathcal{F}_{T_n}\right] \leqslant -c, \text{ for all } n \geqslant 0.
\label{cfinal}
\end{equation}
In conclusion, the stochastic process $Z_n=u(X_n)+cn$ is a super-martingale due to the combination of \eqref{diffz} and \eqref{cfinal}.\\ 
Step 3. In order to apply the optimal stopping theorem described in Proposition 2.6., we need on one hand that $(U(X_n)+cn)_{n\geqslant 0}$ is a super-martingale but also on the other hand that $(U(X_n))_{n \geqslant 0}$ is a non negative sequence. For the first property we could choose $U=u+\kappa$, $u$ being the function introduced in \eqref{solou} and $\kappa$ a constant. For the second property we need to have a non negative sequence, so we have to choose in a suitable way the constant $\kappa$. Let us note that the function $u$ satisfies $u(0)=0$ and is a concave function. So in order to obtain a positive function on the interval $[a_{\gamma,x}, b_{\gamma,x}]$ it suffices to choose $\kappa=-\min(u(b_{\gamma,x}), u(a_{\gamma,x}))$.\\
Consequently we need to study the behavior of $u$ at the frontiers of $[a_{\gamma,x},b_{\gamma,x}]$ that is for $x = b-\epsilon$ and $x = a + \epsilon$. Putting $b_{\gamma} := b_{\gamma, b - \epsilon}$, we obtain
\begin{align*}
u(b_{\gamma}) &= -\frac{2}{\sigma^2} \int^{b_{\gamma}}_0 e^{\frac{2\theta}{\sigma^2}u}\int^u_0 \frac{e^{-\frac{2\theta}{\sigma^2}s}}{(s-a)^2(s-b)^2}ds\, du\\
&=-\frac{2}{\sigma^2} \int^{b_{\gamma}}_0 \frac{e^{-\frac{2\theta}{\sigma^2}s}}{(s-a)^2(s-b)^2} \int^{b_{\gamma}}_s e^{\frac{2\theta}{\sigma^2}u}ds\,du\\
&= -\frac{1}{\theta} \int^{b_{\gamma}}_0 \frac{e^{\frac{2\theta}{\sigma^2}(b_{\gamma, b - \epsilon}-s)}-1}{(s-a)^2(s-b)^2} ds.
\end{align*}
Using Taylor's expansion
\begin{equation*}
e^{\frac{2\theta}{\sigma^2}(b_{\gamma}-s)}-1 = \frac{2\theta}{\sigma^2}(b_{\gamma} - s) + \frac{2 \theta^2 e^{\frac{2\theta \xi}{\sigma^2}}}{\sigma^4} (b_{\gamma} - s)^2,
\end{equation*}
where $\xi$ belongs to $[0,b_{\gamma}-s]$.
Moreover
\begin{equation}
\frac{1}{(s-a)^2(s-b)^2} = \frac{c_1}{(s-a)} + \frac{c_2}{(s-a)^2}+ \frac{c_3}{(b-s)}+ \frac{c_4}{(b-s)^2},
\end{equation}
where $c_i, i\in\{1,2,3,4\}$ are positive constants and $c_2 = c_4 = \displaystyle{\frac{1}{(b-a)^2}}$.

\begin{align*}
u(b_{\gamma, b - \epsilon}) &= -\frac{2}{\sigma^2} \int^{b_{\gamma}}_0 \frac{(b_{\gamma} - s)}{(s-a)^2(s-b)^2} ds -  \frac{2\theta}{\sigma^4} \int^{b_{\gamma}}_0 \frac{e^{\frac{2\theta \xi}{\sigma^2}}(b_{\gamma} - s)^2}{(s-a)^2(s-b)^2} ds \\
&= -\frac{2}{\sigma^2} \Big(c_1I_{0,1,1}-c_2 I_{0,2,1}-c_3 I_{1,0,1}-c_4 I_{2,0,1}\\
&\quad \quad \quad \quad+  \frac{\theta}{\sigma^2} \int^{b_{\gamma}}_0 \frac{e^{\frac{2\theta \xi}{\sigma^2}}(b_{\gamma} - s)^2}{(s-a)^2(b-s)^2} ds\Big),
\end{align*}
where $I_{i,j,k} = \displaystyle{\int^{b_{\gamma}}_0 \frac{(b_{\gamma} - s)^k}{ (b-s)^i(s-a)^j} ds}$.
We can notice that
\begin{equation}
I_{2,0,1} = \log(\gamma\epsilon) - \log(\vert b \vert) + 1 -\frac{\gamma \epsilon}{b} 
\label{I}
\end{equation}
and
\begin{equation}
c_1I_{0,1,1}+c_2 I_{0,2,1}+c_3 I_{1,0,1} = o(1) \text{ as $\epsilon$ tends to $0$.}
\label{equI}
\end{equation}
Let us bound the last integral, using once again the partial fraction decomposition 
\begin{align*}
0 \leqslant \frac{\theta}{\sigma^2} \int^{b_{\gamma}}_0 \frac{e^{\frac{2\theta \xi}{\sigma^2}}(b_{\gamma} - s)^2}{(s-a)^2(b-s)^2} ds \leqslant \frac{\theta}{\sigma^2} e^{\frac{2\theta b}{\sigma^2}}\int^{b_{\gamma}}_0 \frac{(b_{\gamma} - s)^2}{(s-a)^2(b-s)^2} ds\\
= \frac{\theta}{\sigma^2} e^{\frac{2\theta b}{\sigma^2}}\left(c_4 I_{0,1,2}+c_5 I_{0,2,2} +c_6 I_{1,0,2}+ c_7 I_{2,0,2} \right).
\end{align*}
As in the previous computations, it is possible to take an equivalent as $\epsilon$ tends to zero, that is there exists $\delta >0$ such that
\begin{equation}
\frac{2\theta}{\sigma^4} e^{\frac{2\theta b}{\sigma^2}}(c_4 I_{0,1,2}+ c_5 I_{0,2,2} + c_6 I_{1,0,2}+ c_7 I_{2,0,2}) \sim \delta.
\end{equation}
Combining \eqref{I} and \eqref{equI}, and taking an equivalent when $\epsilon$ tends to $0$ leads to state that 
\begin{equation}
u(b_{\gamma, b - \epsilon}) \sim -D \log(\gamma\epsilon) -\tilde{D}-  \delta, \text{ where } D = \displaystyle{\frac{2}{\sigma^2(b-a)^2}}.
\end{equation}
A similar computation on $u(a_{\gamma,a + \epsilon})$ gives us some $\tilde{D}'$ and $\tilde{\delta}$. Setting $\kappa = \max(\tilde{D}+\delta, \tilde{D}'+\tilde{\delta})$ and $U(x) = u(x) - \kappa$ permits to obtain the positivity of the sequence $(U(X_n))_{n\ge 1}$.\\
Step 4. The statement of the theorem is a direct consequence of the optimal stopping theorem Proposition 2.6. If $N_\epsilon$ is almost surely finite, then 
\begin{equation}
\mathbb{E}[N_{\epsilon}] \leqslant \frac{1}{c}\mathbb{E}[U(X_0)] \leqslant D \log(\epsilon).
\end{equation}
In order to finish the proof, it remains to justify that $N_\epsilon$ is almost surely finite. Since $b_{\gamma,x}-x\ge (1-\gamma)\epsilon$ and $x-a_{\gamma,x}\ge (1-\gamma)\epsilon$ for any $x\in[a+\epsilon, b-\epsilon]$, we deduce that there exists a strictly positive lower-bound $d_\epsilon$ such that $d_{X_n}\ge d_\epsilon$ for any $n$. Introducing $(s_n)$ a sequence of independent and identically distributed random variables corresponding to Brownian exits of a unit spheroid, we deduce that $T_n$ is stochastically lower-bounded by
\[
S_n:=\frac{1}{2\theta}\sum_{k=1}^n\log(1+d_\epsilon s_k). 
\] 
Moreover $S_n$ tends to infinity almost surely as $n\to\infty$. By Lemma \ref{TfiniteOU} and by construction, $T_n$ is stochastically inbetween $S_n$ and $\mathcal{T}$ (an almost surely finite random variable) for any $n\le N_\epsilon$. The stopping rule $N_\epsilon$ is therefore almost surely finite. 
\end{proof}

\begin{lemma}
\label{TfiniteOU}
The sequence of cumulative times $(T_n)_{n\geqslant 1}$ appearing in the algorithm  and defined by \eqref{defTnOU} are stochastically smaller than $\mathcal{T}$ the first exit time of the Ornstein-Uhlenbeck process.
\end{lemma}

\begin{proof}
We need to emphasize the link between the markov chain induced by the algorithm, denoted $((T_n, X_n))_{n\in \mathbb{N}}$ with $(T_0, X_0) = (0,0)$,  and a path of the Ornstein-Ulhenbeck process.\\
At the starting point of the Ornstein-Uhlenbeck trajectory, we introduce a spheroid of maximum size contained in the interval $[a,b] \times \mathbb{R}_+$. The intersection of this spheroid and the path corresponds to the point $(t_1,z_1)$. Then  this construction leads us to state that $(t_1,z_1)$ has the same distribution as $(T_1, X_1)$. Hence, from $(t_1,z_1)$ we can construct a maximum size spheroid and consider the intersection $(t_2,z_2)$ between the trajectory after $t_1$ and this second spheroid. Once again we get from the construction that  $(t_2,z_2)$ and $(T_2, X_2)$ are identically distributed. We can therefore step by step build a sequence $((t_n, z_n))_{n\in \mathbb{N}}$ of intersections between the considered trajectory and the spheroids. We obtain that the skeleton of the trajectory $(t_n,z_n)_{n\in \mathbb{N}}$ and the sequence $(T_n,X_n)_{n\in \mathbb{N}}$ are identically distributed.
 By construction, we also note that $t_n \leqslant \mathcal{T}$ for all $n \in \mathbb{N}$, which implies the announced result.
\end{proof}

\subsubsection*{Bounds for the Exit-Time distribution }

Let us now precise the rate of convergence for the algorithm based on the random walk. We should describe how far the outcome of the algorithm and the diffusion exit time are. We recall that the outcome depends on the parameter $\epsilon$. 

\begin{thm}\label{thm:bounds}
We consider $0<\gamma<2$ and $\delta = \epsilon^{\gamma}$. We denote by $F$ the cumulative distribution function of the exit time from the interval $[a,b]$ and $F_{\epsilon}$ the distribution function of the algorithm outcome. Then for any $\rho>1$, there exists $\epsilon_0>0$ such that
\begin{equation}
\left( 1 - \frac{\rho\sqrt{\theta}(\epsilon + \max(\vert a \vert,\vert b\vert) (e^{\theta \delta} -1))}{\sigma\sqrt{ (e^{2\theta \delta} -1) \pi}}\right)F_{\epsilon}(t -\delta) \leqslant F(t) \leqslant F_{\epsilon}(t),
\end{equation}
for all $t\in\mathbb{R}$ and $\epsilon\le \epsilon_0$.
\end{thm}


In other words, the precision of the approximation pointed out in Theorem \ref{thm:bounds} is characterized by the following error bound:
\[
\Xi(\epsilon\,;\theta,\sigma,a,b,\gamma):=\frac{\sqrt{\theta}(\epsilon + \max(\vert a \vert,\vert b\vert) (e^{\theta \delta} -1))}{\sigma\sqrt{ (e^{2\theta \delta} -1) \pi}},\quad \mbox{with}\ \delta=\epsilon^\gamma.
\]
 Figure 5 presents the dependence of this bound with respect to $\epsilon$ and $\theta$,  all other parameters being fixed.

\begin{figure}[H]
   \includegraphics[width=\textwidth]{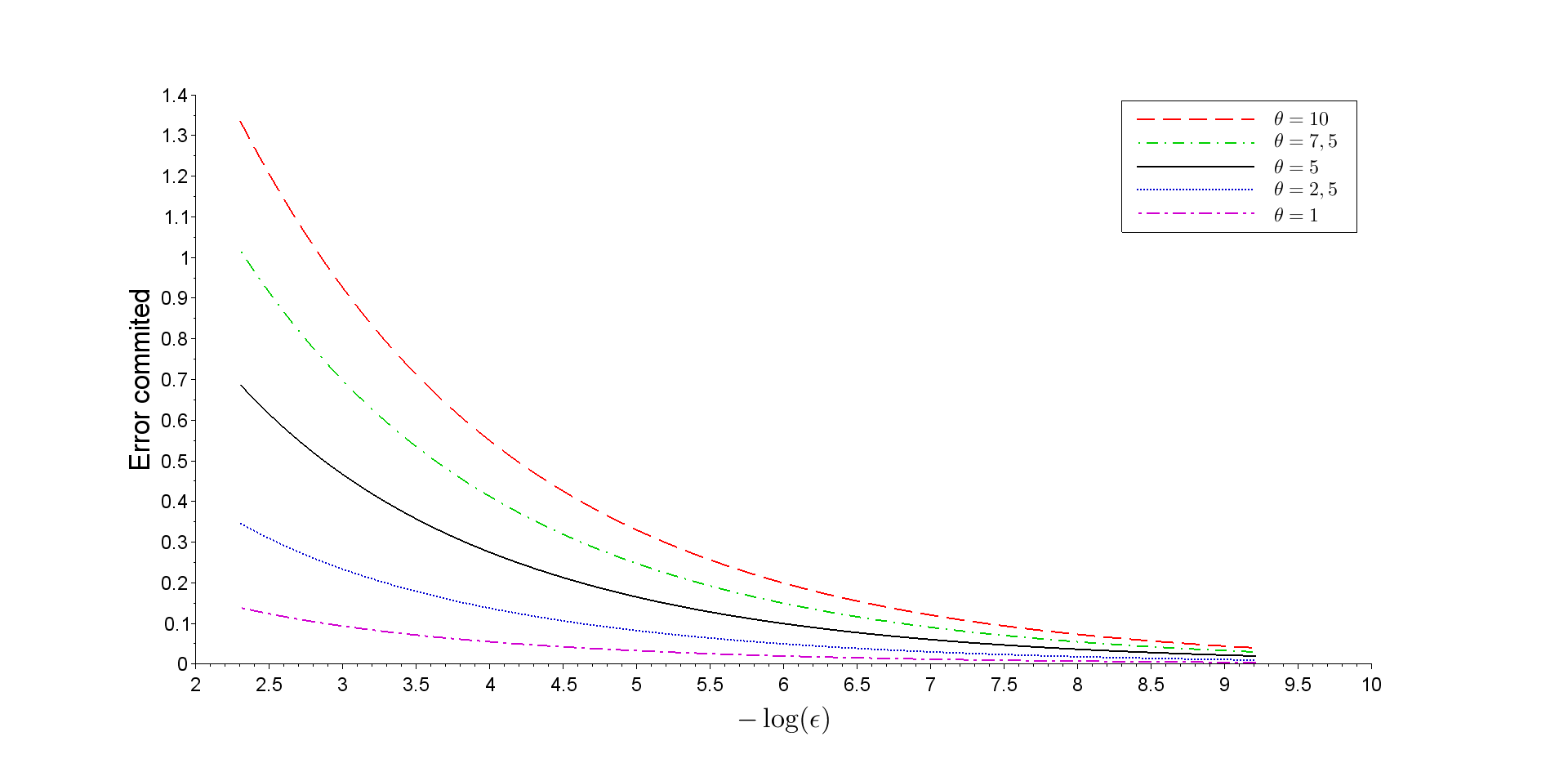}
   \caption{Error bound $\Xi$ versus $\epsilon$ for different values of $\theta$ with $\sigma = 1$, $a =-1$, $b=1$, $\gamma =1$.}
\end{figure}

Such a statement is directly related to properties of the Ornstein-Uhlenbeck process and its strong link with the Brownian motion. 

\begin{proof}
As in Lemma \ref{TfiniteOU}, we build step by step a sequence $((t_n, z_n))_{n\in \mathbb{N}}$ of intersections between the path of the Ornstein-Uhlenbeck process and the spheroids in such a way that the sequences $((t_n,z_n))_{n\ge 0}$ and $((T_n,X_n))_{n\ge 0}$ are identically distributed.\\
If we introduce $N_{\epsilon}$ the stopping time appearing in the stopping procedure of the algorithm  and $\tilde{N}_{\epsilon} = \inf\{n \in \mathbb{N}, z_n \notin [a + \epsilon, b - \epsilon]\}$, the identity in law of those random variables yields. By construction, $t_n \leqslant \mathcal{T}$ for all $n \in \mathbb{N}$, where $\mathcal{T}$ stands for the diffusion first exit time from the interval $[a,b]$. This inequality remains true when $t_n$ is replaced by the random stopping time $t_{\tilde{N}_\epsilon}$.\\
Hence
\begin{align}
1 - F(t) &= \mathbb{P}(\mathcal{T}>t) \nonumber\\
&=\mathbb{P}(\mathcal{T}>t, t_{\tilde{N}_{\epsilon}}\leqslant t -\delta) + \mathbb{P}(\mathcal{T}>t, t_{\tilde{N}_{\epsilon}}> t -\delta)\nonumber\\
&\leqslant \mathbb{P}(\mathcal{T}>t, t_{\tilde{N}_{\epsilon}}\leqslant t -\delta) + \mathbb{P}(t_{\tilde{N}_{\epsilon}} > t - \delta)\nonumber\\
&\leqslant \mathbb{P}(\mathcal{T}>t, t_{\tilde{N}_{\epsilon}}\leqslant t -\delta) +1 - F_{\epsilon}(t- \delta).
\label{majou}
\end{align}
We focus on the first term of this upper bound. Using the strong Markov property, we obtain
\begin{equation}
\mathbb{P}(\mathcal{T}>t, t_{\tilde{N}_{\epsilon}}\leqslant t -\delta) \leqslant F_{\epsilon}(t-\delta) \sup\limits_{y \in [a,a+ \epsilon] \cup [b-\epsilon,b]} \mathbb{P}_y(\mathcal{T}>\delta).
\label{markov}
\end{equation}
For any $y \in [a,a+ \epsilon] \cup [b-\epsilon,b]$ we write
 \begin{equation*}
 \mathbb{P}_y(\mathcal{T}>\delta) = \mathbb{P}_y(\mathcal{T}_a> \delta, \mathcal{T}_a<\mathcal{T}_b) + \mathbb{P}_y(\mathcal{T}_b> \delta, \mathcal{T}_b<\mathcal{T}_a).
 \end{equation*}
 We first consider the case $y \in [b-\epsilon, b]$, the previous inequality becomes
\begin{equation}
 \mathbb{P}_y(\mathcal{T}>\delta) \leqslant \mathbb{P}_y(\mathcal{T}_a<\mathcal{T}_b) + \mathbb{P}_y(\mathcal{T}_b> \delta).
 \label{decomproba}
 \end{equation}
In order to handle with the first term in the right hand side, we introduce $s$ the scale function of the O.-U.-process:
\begin{equation}
s(x) = e^{\frac{\theta}{\sigma^2}a^2}\int_a^x e^{\frac{\theta}{\sigma^2}u^2}du, \quad x \in [a,b].
\label{scalou}
\end{equation} 
It has been shown in Karatzas, 5.5 \cite{kar} that 
\begin{equation}
\mathbb{P}_y(\mathcal{T}_a<\mathcal{T}_b) = \frac{s(b) - s(y)}{s(b) - s(a)} =\frac{\int_{y}^b e^{\frac{\theta}{\sigma^2}u^2}du}{\int_a^b e^{\frac{\theta}{\sigma^2}u^2}du}.
\label{comparscale}
\end{equation}
Since $ y \in [b- \epsilon,b]$ and since the integrated function is non negative and increasing we obtain 
\begin{equation}
\mathbb{P}_y(\mathcal{T}_a<\mathcal{T}_b) \leqslant \frac{\int_{b- \epsilon}^b e^{\frac{\theta}{\sigma^2}u^2}du}{\int_a^b e^{\frac{\theta}{\sigma^2}u^2}du}\leqslant \frac{\epsilon \, e^{\frac{\theta}{\sigma^2}b^2}}{\int_a^b e^{\frac{\theta}{\sigma^2}u^2}du} =: \epsilon \, C_{a,b}.
\label{majerreurun}
\end{equation}
We now focus on the second term in the r.h.s. of \eqref{decomproba}: $\mathbb{P}_y(\mathcal{T}_b> \delta) \leqslant \mathbb{P}_{b-\epsilon}(\mathcal{T}_b> \delta)$ for all $y \in [b-\epsilon, b]$. We denote by $\tilde{X}$ the Ornstein-Uhlenbeck process starting in $b - \epsilon$. We obtain 
\begin{align*}
\left\lbrace \mathcal{T}_b > \delta \right\rbrace &=\left\lbrace\sup\limits_{u \in [0,\delta]} \tilde{X}_{u} <b\right\rbrace = \left\lbrace \forall u \in [0,\delta], \tilde{X}_u<b\right\rbrace \\
&=\left\lbrace  (b-\epsilon)e^{-\theta u} + \frac{\sigma}{\sqrt{2\theta}}e^{-\theta u} V_{e^{2\theta u} -1}<b, \forall u \in [0,\delta] \right\rbrace \\
&=\left\lbrace  V_{s}<\frac{\sqrt{2\theta}}{\sigma}( b \sqrt{s+1}  -(b-\epsilon)), \forall s \in [0,e^{2\theta \delta} -1]\right\rbrace\\
&=\left\lbrace V_{s}<\frac{\sqrt{2\theta}}{\sigma}( b (\sqrt{s+1}  -1)+\epsilon)), \forall s \in [0,e^{2\theta \delta} -1]\right\rbrace\\
&\subset\left\lbrace V_s < \frac{\sqrt{2\theta}}{\sigma}(\epsilon + \max(0,b) (e^{\theta \delta} -1)), \forall s \in [0,e^{2\theta \delta} -1] \right\rbrace. \\
\end{align*}
Let us assume that $b>0$. In this case, the following asymptotic property holds:
\begin{align*}
\mathbb{P}_{b - \epsilon}(\mathcal{T}_b> \delta) &= \mathbb{P}_0\left(\sup\limits_{s \in [0,e^{2\theta \delta} -1]}  V_t  < \frac{\sqrt{2\theta}}{\sigma}(\epsilon + b (e^{\theta \delta} -1))\right)\\
&= \mathbb{P}_0\left(2\vert V_{e^{2\theta \delta} -1} \vert < \frac{\sqrt{2\theta}}{\sigma}(\epsilon + b (e^{\theta \delta} -1))\right)\leqslant\frac{\sqrt{\theta}(\epsilon + b (e^{\theta \delta} -1))}{\sigma\sqrt{ (e^{2\theta \delta} -1) \pi}}.
\end{align*}
Using the particular form of $\delta = \epsilon^{\gamma}$, we obtain
\begin{equation*}
\frac{\sqrt{\theta}(\epsilon + b (e^{\theta \delta} -1))}{\sigma\sqrt{ (e^{2\theta \delta} -1) \pi}} \sim \frac{1}{\sigma \sqrt{2 \pi}}(\epsilon^{1 - \frac{\gamma}{2}} + b \theta \epsilon^{\frac{\gamma}{2}})\quad \mbox{as}\ \epsilon\to 0.
\end{equation*}
A similar bound can be obtained for b negative and also for $y\in[a,a+\epsilon]$.\\ 
Finally combining this result with \eqref{majou}, \eqref{markov} and  \eqref{majerreurun} leads to the announced statement.
\end{proof}

\begin{rem}
Let us note that all the results presented so far, that is the efficiency of the algorithm and the convergence rate, concern the family of Ornstein-Uhlenbeck processes with parameter $\mu=0$ in (2.1). It is straightforward to extend the statements to the general case: it suffices to replace the interval $[a,b]$ by a time-dependent interval $[a-\mu(1-e^{-\theta t}),b-\mu(1-e^{-\theta t})]$.
\end{rem}
\bibliographystyle{plain}

\begin{thebibliography}{}

\end{thebibliography}


\begin{thebibliography}{10}

\bibitem{baldi-caramellino-iovino}
P. Baldi and L. Caramellino and M.G. Iovino.
\newblock Pricing general barrier options: a numerical approach using sharp large deviations.
\newblock {\em Math. Finance  9(4): 293--322}, 1999.         

\bibitem{binder-braverman}
I. Binder and M. Braverman.
\newblock The rate of convergence of the walk on spheres algorithm.
\newblock {\em Geom. Funct. Anal. 22(3):558-587}, 2012.

\bibitem{buonocore-al}
A. Buonocore and A.G. Nobile and L.M. Ricciardi.  
\newblock A new integral equation for the evaluation of first-passage-time probability densities.
\newblock  {\em Adv. in Appl. Probab.,  19(4): 784--800}, 1987.

\bibitem{Deaconu-Herrmann}
M. Deaconu and S. Herrmann.
\newblock Simulation of hitting times for bessel processes with noninteger
  dimension.
\newblock {\em Bernoulli 23(4B):3744-3771}, 2017.

\bibitem{herrmann-zucca-exact}
S. Herrmann and C. Zucca
\newblock Exact simulation of the first-passage time of diffusions.
\newblock {\em J. Sci. Comput., 79(3):1477-1504}, 2019

\bibitem{kar}
I. Karatzas and S.~E. Shreve.
\newblock {\em Brownian motion and stochastic calculus}, volume 113 of {\em
  Graduate Texts in Mathematics}.
\newblock Springer-Verlag, New York, second edition, 1991.

\bibitem{lerche}
H.~R. Lerche.
\newblock {\em Boundary crossing of {B}rownian motion}, volume~40 of {\em
  Lecture Notes in Statistics}.
\newblock Springer-Verlag, Berlin, 1986.

\bibitem{mot}
M. Motoo.
\newblock Some evaluations for continuous {M}onte {C}arlo method by using
  {B}rownian hitting process.
\newblock {\em Ann. Inst. Statist. Math. Tokyo}, 11:49--54, 1959.

\bibitem{mul}
M.~E. Muller.
\newblock Some continuous {M}onte {C}arlo methods for the {D}irichlet problem.
\newblock {\em Ann. Math. Statist.}, 27:569--589, 1956.

\bibitem{norris}
J. R. Norris.
\newblock {\em Markov Chains}.
\newblock Cambridge University Press, 1996.

\bibitem{rev}
D. Revuz and M. Yor.
\newblock {\em Continuous martingales and {B}rownian motion}, volume 293 of
  {\em Grundlehren der Mathematischen Wissenschaften [Fundamental Principles of
  Mathematical Sciences]}.
\newblock Springer-Verlag, Berlin, third edition, 1999.

\bibitem{sab1}
K.~K. Sabelfeld and N.~A. Simonov.
\newblock {\em Random walks on boundary for solving {PDE}s}.
\newblock VSP, Utrecht, 1994.

\bibitem{sab2}
K.~K. Sabelfeld.
\newblock {\em Monte {C}arlo methods in boundary value problems}.
\newblock Springer Series in Computational Physics. Springer-Verlag, Berlin,
  1991.

\bibitem{sacerdote-telve-zucca}
L. Sacerdote, O. Telve, and C. Zucca.
\newblock Joint densities of first hitting times of a diffusion process through
  two time-dependent boundaries.
\newblock {\em Adv. in Appl. Probab., 46(1):186-202}, 2014.

\end{thebibliography}

\end{document}